\documentclass[10pt]{amsart}
\usepackage{fullpage}
\usepackage[breaklinks]{hyperref}

\newtheorem{lemma}{Lemma}[section]
\newtheorem{theorem}[lemma]{Theorem}

\theoremstyle{definition}

\newtheorem*{ack}{Acknowledgements}
\newtheorem{remark}[lemma]{Remark}

\newcommand{\R}{\mathbb{R}}
\newcommand{\C}{\mathbb{C}}
\newcommand{\Z}{\mathbb{Z}}
\renewcommand{\S}{\mathbb{S}}
\newcommand{\ep}{\varepsilon}

\newcommand{\End}{\operatorname{End}}

\newcommand{\tl}{\tilde}
\newcommand{\wt}{\widetilde}

\newcommand{\J}{\mathcal{J}}
\newcommand{\X}{X_{\lambda}}

\newcommand{\dt}{\partial_{t}}
\newcommand{\dth}{\partial_{\theta}}
\newcommand{\dph}{\partial_{\phi}}
\newcommand{\drh}{\partial_{\rho}}

\newcommand{\ds}{\partial_{s}}

\def\bp#1{\left(#1\right)}
\def\br#1{\left\{#1\right\}}

\def\abs#1{\left|#1\right|}

\begin{document}
\title[Nonunique asymptotic limit]{Finite-energy pseudoholomorphic planes with multiple asymptotic limits}
\author[R.\ Siefring]{Richard Siefring}
\address{Fakult\"at f\"ur Mathematik \\ Ruhr-Universit\"at Bochum \\
     44780 Bochum \\
    Germany}
\urladdr{\url{http://homepage.ruhr-uni-bochum.de/richard.siefring}}
\email{\href{mailto:richard.siefring@ruhr-uni-bochum.de}{richard.siefring@ruhr-uni-bochum.de}}
\date{September 30, 2016}

\dedicatory{
Dedicated to Helmut Hofer on the occasion of his 60th birthday,
and in warm remembrance of Kris Wysocki.
}

\begin{abstract}
It's known from \cite{hwz:prop1, hwz:prop4, bourgeois} that 
in a contact manifold equipped with either a nondegenerate or Morse-Bott
contact form,
a finite-energy pseudoholomorphic curve will 
be asymptotic
at each of its nonremovable punctures
to a single periodic orbit of the Reeb vector field and
that the convergence is exponential.
We provide examples here to show that this need not be the case
if the contact form is degenerate.
More specifically, 
we show that on any contact manifold $(M, \xi)$ with cooriented contact structure
one can choose a contact form $\lambda$ with $\ker\lambda=\xi$
and a compatible complex structure $J$ on $\xi$
so that for the associated $\R$-invariant almost complex structure $\tl J$ on $\R\times M$
there exist families of embedded finite-energy $\tl J$-holomorphic cylinders and planes having
embedded tori as limit sets.
\end{abstract}

\maketitle

\tableofcontents


\section{Introduction and statement of results}
The study of punctured pseudoholomorphic curves in symplectizations of contact manifolds
was introduced by Hofer in \cite{hofer93}.
Specifically, considering a contact manifold $(M, \xi=\ker\lambda)$,
Hofer introduced a class of
$\R$-invariant almost complex structures
and a notion of energy for a pseudoholomorphic map $\tl u=(a, u):\C\to \R\times M$ 
and showed that if the energy of a pseudoholomorphic plane is finite, 
then
there are sequences
$s_{k}\to\infty$
so that the sequence of loops
\[
t\in \S^{1}\approx \R/\Z \mapsto u(e^{2\pi(s_{k}+it)})
\]
converge in $C^{\infty}(\S^{1}, M)$ to a periodic orbit $\gamma$ of the Reeb vector field of the contact form $\lambda$.

In \cite[Theorem 1.2/1.3]{hwz:prop1}, Hofer, Wysocki, and Zehnder further show
that if the periodic orbit $\gamma$ is nondegenerate, then the maps
$u(s):\S^{1}\to M$ defined by $u(s)(t)=u(e^{s+it})$ satisfy
\[
\lim_{s\to\infty}u(s)=\gamma \text{ in $C^{\infty}(\S^{1}, M)$} 
\]
and in fact the convergence is exponential
\cite[Theorem 1.4]{hwz:prop1}.
There, immediately following the statement of
Theorem 1.2,
the authors mention that they expect this need not be the case
in the event that the periodic orbit $\gamma$ is degenerate, but that they didn't know of an
explicit example.
To date no examples have appeared in the literature, and whether or not it is possible for a
finite-energy plane to have multiple periodic orbits as asymptotic limits has remained an open
question.\footnote{In fact, a claimed proof that
no such examples exist has appeared in a recent (now-withdrawn) arXiv preprint.}
We present some examples here.  The examples we construct can be localized to any arbitrarily small
neighborhood of a standard model of a transverse knot and since transverse knots exist in abundance in any
contact manifold, we can prove the following.

\begin{theorem}\label{t:main-theorem}
Let $(M, \xi)$ be a contact manifold.  Then there exists a contact form $\lambda$ on $M$
and a compatible complex structure $J$ on $\xi$ so that there exist finite-energy pseudoholomorphic planes and
cylinders for the data $(\lambda, J)$ whose limit sets have image diffeomorphic to the $2$-torus.
\end{theorem}

We give a brief outline of what follows.  In Section \ref{s:background} we begin by recalling
some basic notions from contact geometry and pseudoholomorphic curves.  
Then, in Section \ref{s:prequant}, we explain a correspondence between gradient flow lines
on exact symplectic manifolds and pseudoholomorphic cylinders
in contact manifolds constructed as circle bundles over those symplectic manifolds.
From this construction it is clear that one can construct pseudoholomorphic cylinders
having more than one limit orbit by constructing gradient flow lines in a symplectic manifold
having an alpha or omega limit set consisting of more than a single point.
To this end, we construct in Section \ref{s:gradient} a function on the cylinder
$\R\times S^{1}$
which will have the circle $\br{0}\times S^{1}$
as the omega limit set of any nontrivial gradient flow line with respect to
any Riemannian metric
and
which
can be chosen to be linear in the $\R$-variable and independent of the $S^{1}$-variable
outside of any desired neighborhood
of $\br{0}\times S^{1}$.
Finally, in Section \ref{s:main-proof}, we
apply the results of 
Sections \ref{s:prequant} and \ref{s:gradient} to
construct finite-energy
pseudoholomorphic cylinders and planes having
tori as limit sets.
We comment that while the construction of a pseudoholomorphic cylinder
with tori as limit sets is a straightforward application of the results in
Sections \ref{s:prequant} and \ref{s:gradient},
applying these results to construct
a plane with multiple limit orbits
is a bit trickier and requires finding a situation where these results can be applied
to construct a cylinder with a removable singularity.

We close this section with a remark about notation.
In most of what follows we find it convenient to consider the circle as $\R/2\pi\Z$,
although at some points ---
specifically when considering domains of pseudoholomorphic cylinders or
periodic orbits ---
we will find it more convenient to consider the circle to be $\R/\Z$.  To avoid ambiguity we 
will use the notations $S^{1}=\R/2\pi\Z$ and $\S^{1}=\R/\Z$ to distinguish between the two.

\begin{ack}
I would like to thank Luis Diogo and Urs Frauenfelder
for helpful discussions.
I also gratefully acknowledge financial support from DFG grant BR 5251/1-1
and the Fakult\"at f\"ur Mathematik at the Ruhr-Universit\"at Bochum.
\end{ack}

\section{Pseudoholomorphic curves in contact manifolds}\label{s:background}

Here we recall some basic notions, primarily for the purpose of fixing notation.
Let $M$ be an oriented $(2n+1)$-dimensional manifold.  A $1$-form
$\lambda$ is said to be a contact form on $M$ if
\begin{equation}\label{e:contact-condition}
\text{$\lambda\wedge d\lambda^{n}$ is a nowhere vanishing.}
\end{equation}
A contact form on $M$ determines a splitting
\begin{equation}\label{e:splitting}
TM=\R\X\oplus\xi
\end{equation}
where $\xi=\ker\lambda$ is a hyperplane distribution, called the \emph{contact structure}, and
$\X$ is the \emph{Reeb vector field}, defined by
\[
i_{\X}d\lambda=0 \qquad\text{ and }\qquad i_{\X}\lambda=1.
\]
We note that \eqref{e:contact-condition} implies that $d\lambda$ restricts to a nondegenerate form on
$\xi$ and thus $(\xi, d\lambda)$ is a symplectic vector bundle over $M$.

We recall that if $\lambda$ is a contact form on $M$ and
$f:M\to\R$ is a smooth function, then
$e^{f}\lambda$ is also a contact form since
$d(e^{f}\lambda)=e^{f}(df\wedge\lambda+d\lambda)$
and hence
\[
(e^{f}\lambda)\wedge d(e^{f}\lambda)^{n}=e^{(n+1)f}\lambda\wedge d\lambda.
\]
We note for later reference that a straightforward computation shows that
the Reeb vector field for the contact form $e^{f}\lambda$ is related to the
Reeb vector field for $\lambda$ by
\begin{equation}\label{e:reeb-change-1}
X_{e^{f}\lambda}=e^{-f}\bp{\X-X_{f}}
\end{equation}
where $X_{f}$ is the unique section of $\xi$ satisfying
\begin{equation}\label{e:reeb-change-2}
i_{X_{f}}d\lambda=-df+df(\X)\lambda.
\end{equation}
That there is a unique section $X_{f}$ of $\xi$ satisfying \eqref{e:reeb-change-2} follows from
nondegeneracy of $d\lambda$ on $\xi$ and the fact that both sides of \eqref{e:reeb-change-2} vanish
on $\X$.

Given a symplectic vector bundle
$(E, \omega)$ over a given manifold $W$, a complex structure $J\in\End(E)$ is said
to be \emph{compatible} with $\omega$ if the section of
$E^{*}\otimes E^{*}$ defined by
$g_{J}:=\omega(\cdot, J\cdot)$ is symmetric and positive definite on $E$.
It is well known that the space of such $J$ is nonempty and contractible
(see e.g.\ the discussion following Proposition 5 in Section 1.3 of \cite{hoferzehnder}).
Given a contact manifold, $(M, \xi=\ker\lambda)$, we then define the set
$\J(M, \xi)$ to be the set of complex structures on $\xi$ compatible with
$d\lambda|_{\xi\times\xi}$.  We observe that if a complex structure $J\in\End(\xi)$
is compatible with $d\lambda$, then it is also compatible with
$d(e^{f}\lambda)$ since
\[
d(e^{f}\lambda)-e^{f}d\lambda=e^{f}df\wedge\lambda
\]
which vanishes on $\xi\times\xi=\ker\lambda\times\ker\lambda$.  Therefore, the set
$\J(M, \xi)$ depends only on a choice of conformal symplectic structure on  $\xi$,
and not on the choice of a specific contact form inducing that structure.

Given a manifold $M$ with contact form $\lambda$ and a compatible $J$, we can extend
$J$ to an $\R$-invariant almost complex structure $\tl J$ on $\R\times M$ by requiring
\begin{equation}\label{e:R-invariant-J}
\tl J\partial_{a}=\X \qquad\text{ and }\qquad \tl J|_{\pi_{M}^{*}\xi}=\pi_{M}^{*}J
\end{equation}
with $a$ the coordinate along $\R$ and $\pi_{M}:\R\times M\to M$ the coordinate projection.  We consider
quintuples $(\Sigma, j, \Gamma, a, u)$ where $(\Sigma, j)$ is a closed Riemann surface,
$\Gamma\subset\Sigma$ is a finite set, called the set of \emph{punctures},
and $a:\Sigma\setminus\Gamma\to\R$ and
$u:\Sigma\setminus\Gamma\to M$ are smooth maps.  We say such a quintuple is
\emph{pseudoholomorphic map for the data $(\lambda, J)$ on $M$}
if $\tl u=(a, u):\Sigma\setminus\Gamma\to\R\times M$ satisfies the equation
\begin{equation}\label{e:j-hol}
d\tl u\circ j=\tl J\circ \tl u
\end{equation}
or, equivalently, if $u$ and $a$ satisfy
\begin{equation}\label{e:j-hol-M}
\begin{gathered}
\pi_{\lambda}\circ du\circ j=J\circ \pi_{\lambda}\circ du \\
u^{*}\lambda\circ j=da
\end{gathered}
\end{equation}
where $\pi_{\lambda}:TM\approx\R\X\oplus\xi\to\xi$ is the projection of $TM$ onto $\xi$ along $\X$.
The \emph{Hofer energy} $E(u)$ of a pseudoholomorphic map $(\Sigma, j, \Gamma, a, u)$ is defined by
\begin{equation}\label{e:hofer-energy-defn}
E(u)=\sup_{\varphi\in\Xi}\int_{\Sigma\setminus\Gamma}\tl u^{*}d(\varphi\lambda)=
\sup_{\varphi\in\Xi}\int_{\Sigma\setminus\Gamma}d(\varphi(a) u^{*}\lambda)
\end{equation}
where $\Xi\subset C^{\infty}(\R, [0, 1])$ is the set of smooth functions
$\varphi:\R\to[0, 1]$ with $\varphi'(t)\ge 0$ for all $s\in\R$,
$\lim_{s\to-\infty}\varphi(s)=0$, and $\lim_{s\to\infty}\varphi(s)=1$.

To each puncture in a pseudoholomorphic map we will a assign a quantity called the mass of the puncture.
First, we will call a holomorphic embedding
$\psi:[0, +\infty)\times \S^{1}\subset\C/i\Z\to\Sigma\setminus\Gamma$ a
\emph{holomorphic cylindrical coordinate system} around $z_{0}\in\Gamma$ if
$\lim_{s\to\infty}\psi(s, t)=z_{0}$.
Given a holomorphic cylindrical coordinates $\psi$ around $z_{0}\in\Gamma$,
we consider the family of loops $v(s)=(u\circ\psi)(s, \cdot):\S^{1}\to M$ and
define the
\emph{mass $m(z_{0})$ of the puncture $z_{0}$}  by
\begin{equation}\label{e:mass}
m(z_{0})=\lim_{s\to\infty}\int_{\S^{1}}v(s)^{*}\lambda.
\end{equation}
The limit in this definition is well-defined as a result the compatibility of $J$ with $d\lambda$.
Indeed, for $s_{1}>s_{0}$ we apply Stokes' theorem to compute
\begin{align}
\int_{\S^{1}}v(s_{1})^{*}\lambda-\int_{\S^{1}}v(s_{0})^{*}\lambda
&=\int_{[s_{0}, s_{1}]\times \S^{1}}(u\circ\psi)^{*}d\lambda  \label{e:mass-stokes} \\
&=\int_{[s_{0}, s_{1}]\times \S^{1}}d\lambda(u_{s}, u_{t})\,ds\wedge dt \notag \\
&=\int_{[s_{0}, s_{1}]\times \S^{1}}d\lambda(\pi_{\lambda}(u_{s}),\pi_{\lambda}(u_{t}))\,ds\wedge dt & i_{\X}d\lambda=0 \notag \\
&=\int_{[s_{0}, s_{1}]\times \S^{1}}d\lambda(\pi_{\lambda}(u_{s}),J\pi_{\lambda}(u_{s}))\,ds\wedge dt  & \eqref{e:j-hol-M} \notag
\end{align}
and we observe the integrand in the final line above is nonnegative by compatibility of $J$ with $d\lambda$.
Thus the integral in the definition \eqref{e:mass} of mass is an increasing function of $s$, which lets
us conclude the limit is well-defined (although possibly infinite).
It can, moreover, be shown that the mass is independent of the choice of holomorphic
cylindrical coordinates near $z_{0}$.

If is a straightforward exercise using \eqref{e:mass-stokes} and definition of Hofer energy to show
that if a pseudoholomorphic map has finite Hofer energy, then all punctures have finite mass.
Furthermore, punctures with mass $0$ can be shown to be removable, that is, one can find a
pseudoholomorphic extension of the map $\tl u$ over any puncture with mass $0$
(see \cite[pgs.\ 272-3]{hwz:prop2}).
The behavior near punctures with nonzero mass is described by the following now well-known theorem of
Hofer from \cite{hofer93}.\footnote{
Hofer only considers planes in \cite{hofer93} and proves the slightly weaker statement that there
exists a sequence $s_{k}\to\infty$ so that corresponding loops
$u\circ\psi(s_{k}, \cdot)$ converge to a periodic orbit,
but the generalization of the proof to the result we state here is straightforward.
In the survey \cite[Theorem 3.2]{hwz-survey}, the appropriate result is proven for a
general pseudoholomorphic half-cylinder, albeit under a different notion of energy.
The fact that this different notion of energy implies finite
Hofer energy as defined by \eqref{e:hofer-energy-defn} is addressed in Theorem 5.1
of the same paper.
}

\begin{theorem}
Let $M$ be a compact manifold equipped with a contact form $\lambda$ and compatible
complex structure $J\in\J(M, \xi)$ on $\xi=\ker\lambda$.
Let $(\Sigma, j, \Gamma, a, u)$ be a solution to
\eqref{e:j-hol-M} and assume that $z_{0}\in\Gamma$ has mass $m(z_{0})=T\ne 0$.
Then for every holomorphic cylindrical coordinate system
$\psi:[0, \infty)\times \S^{1}\to\Sigma\setminus\Gamma$
around $z_{0}$, and every sequence $s_{k}\to\infty$ there exists a subsequence
$s_{k_{j}}$ and a smooth map $\gamma:\S^{1}=\R/\Z\to M$
with $\dot\gamma=T\cdot\X\circ\gamma$ so that
the sequence of loops $u\circ\psi(s_{k_{j}}, \cdot):\S^{1}\to M$ converge in
$C^{\infty}(\S^{1}, M)$ to $\gamma$.
\end{theorem}

We will refer to the collection of periodic orbits obtained as limits of a given finite-energy
pseudoholomorphic map
as the \emph{limit set} of that map.
As mentioned in the introduction, it can be shown under some suitable nondegeneracy assumptions
that a puncture has a unique periodic orbit (up to reparametrization) in its limit set
and that the convergence to that orbit is exponential \cite{hwz:prop1, hwz:prop4, HWZ:planes, bourgeois, mora}.
In the absence of nondegeneracy however,
it has until now remained an open question
whether it's possible for a finite-energy pseudoholomorphic map to have more than one periodic orbit
in the asymptotic limit set of a given puncture.

\section{Prequantization spaces, gradient flows, and pseudoholomorphic cylinders}\label{s:prequant}

In this section we explain a correspondence between gradient
flows on symplectic manifolds and certain
pseudoholomorphic cylinders in an associated prequantization space, that is,
a contact manifold constructed as a principal
$S^{1}$-bundle over the given symplectic manifold
with the contact structure being given as the horizontal distribution determined by an
appropriate connection on the bundle.
For simplicity we focus on the case of trivial $S^{1}$-bundles over exact symplectic manifolds,
since that is all we require for the proof of our main theorem,
but we point out that the construction of pseudoholomorphic cylinders
in a prequantization space from gradient flow lines in the base
can be generalized to any prequantization space.

Let $(W, \omega=d\beta)$ be an exact symplectic manifold and 
consider
$S^{1}(\approx \R/2\pi\Z)\times W$ equipped with the $1$-form
\begin{equation}\label{e:prequant-form}
\lambda=d\theta+\pi^{*}\beta.
\end{equation}
where $\pi:S^{1}\times W\to W$ is the canonical projection onto the second factor.
The $1$-form $\lambda$ defined in this way is a contact form on $S^{1}\times W$ since
\begin{align*}
\lambda\wedge (d\lambda)^{n}
&=(d\theta+\pi^{*}\beta)\wedge \pi^{*}\omega^{n} \\
&=d\theta\wedge\pi^{*}\omega^{n}>0.
\end{align*}
We will refer to a pair $(S^{1}\times W, d\theta+\pi^{*}\beta)$
consisting of a trivial $S^{1}$-bundle and a contact form arising in this way
as a \emph{prequantization space} over the symplectic manifold $(W, \omega=d\beta)$.

We observe that the Reeb vector field of the contact form \eqref{e:prequant-form} is given by
$\dth$ and hence the splitting \eqref{e:splitting} induced on $TM$ by the contact form
is given by
\[
T(S^{1}\times W)\approx TS^{1}\oplus\xi.
\]
Thus $\xi$ is an $S^{1}$-invariant horizontal distribution of the bundle $S^{1}\times W\to W$ which gives us a
one-to-one correspondence between the space $\Gamma(TW)$ of vector fields on $W$ and the space
$\Gamma_{S^{1}}(\xi)$ of
$S^{1}$-invariant sections of the contact structure $\xi$.
This correspondence is given explicitly by the maps
\begin{equation}\label{e:horizontal-lift}
X\in T_{p}W\mapsto \tl X:=-\beta(X)\,\dth +X \in \xi_{(\theta, p)}\subset T_{(\theta, p)}(S^{1}\times W)
\end{equation}
and
\[
\tl Y\in \xi_{(\theta, p)}\mapsto d\pi(\tl Y)\in T_{p}W,
\]
where the plus sign in \eqref{e:horizontal-lift} is to be interpreted relative to the natural splitting
\[
T_{(\theta, p)}(S^{1}\times W)\approx T_{\theta}S^{1}\oplus T_{p}W
\]
arising from the canonical projection onto the factors of the Cartesian product.
The correspondence between vector fields on $W$ and $S^{1}$-invariant sections of $\xi$ generalizes
to arbitrary tensor fields on $W$.  In particular an endomorphism $A\in\End(TW)\approx TW^{*}\otimes TW$
of the form
\[
A=\sum_{i}\alpha_{i}\otimes X_{i}
\]
lifts to an $S^{1}$-invariant endomorphism $\tl A\in\End(\xi)\approx\xi^{*}\otimes\xi$ given by
\[
\tl A=\sum_{i}\pi^{*}\alpha_{i}\otimes \tl X_{i}.
\]
Equivalently, we can define $\tl A$ to be the unique section of $\End(\xi)$ satisfying
\[
\tl A\tl X=\wt{AX}
\]
for every vector field $X$ on $TW$.

We define $\J(W, \omega)$ to be the set of almost complex structures on $W$ compatible with
the symplectic form $\omega$, that is, those $j\in\End(TW)$ which square to negative the identity
and for which $g_{j}:=\omega(\cdot, j\cdot)$ is a Riemannian metric on $W$.
According to the remarks of the previous paragraph, $j$ lifts to an $S^{1}$-invariant
endomorphism $\tl j$ of $\xi$ characterized by
\[
\tl j\tl X=\wt{jX}
\]
for every vector field $X$ on $W$.  From this equation together with the linearity of the map
$X\mapsto\tl X$ and the fact that
\[
d\lambda=d(\pi^{*}\beta)=\pi^{*}(d\beta)=\pi^{*}\omega
\]
it follows that the $S^{1}$-invariant lift $\tl j\in\End(\xi)$ of a compatible almost complex structure
$j\in\J(W, \omega)$ on $W$ is an element of $\J(S^{1}\times W, \xi)$, i.e.\
a complex structure on $\xi$ compatible with $d\lambda$.

Given a choice of compatible $j\in\J(W, \omega)$
we can associate two vector fields on $W$ to any smooth real-valued function
$f$ on $W$:
the Hamiltonian vector field $X_{f}$ and the gradient $\nabla f$ defined respectively
by
\[
i_{X_{f}}\omega=-df \qquad\text{ and }\qquad g_{j}(\nabla f, \cdot)=df.
\]
These vector fields are related by the equations
\[
X_{f}=j\nabla f \qquad\text{ and }\qquad \nabla f=-jX_{f}
\]
since we can use the definition of $g_{j}$ and the antisymmetry of $\omega$ to compute
\[
i_{j\nabla f}\omega=\omega(j\nabla f, \cdot) 
=-\omega(\cdot, j\nabla f)
=-g_{j}(\cdot, \nabla f)
=-df.
\]
From the observations of the previous paragraph, the respective $S^{1}$-invariant lifts of
$\tl j$, $\wt{\nabla f}$, and $\wt{X_{f}}$ of $j$, $\nabla f$ and $X_{f}$ satisfy
\begin{equation}\label{e:ham-grad-lift}
\wt{X_{f}}=\tl j\wt{\nabla f} \qquad\text{ and }\qquad \wt{\nabla f}=-\tl j\wt{X_{f}}.
\end{equation}

Continuing to let $f:W\to\R$ denote a smooth function on $W$, we
can pull $f$ back to an $S^{1}$-invariant smooth function $\pi^{*}f$ on
$S^{1}\times W$ and consider the contact form $\lambda_{f}$ defined by
\[
\lambda_{f}=e^{\pi^{*}f}\lambda=e^{\pi^{*}f}(d\theta+\pi^{*}\beta).
\]
Since
\begin{align*}
i_{\tl X_{f}}d\lambda
&=i_{\tl X_{f}}\pi^{*}\omega \\
&=\pi^{*}(i_{X_{f}}\omega) \\
&=-\pi^{*}df \\
&=-d(\pi^{*}f)+d(\pi^{*}f)(\dth)\lambda
\end{align*}
it follows from \eqref{e:reeb-change-1}-\eqref{e:reeb-change-2} that
\begin{equation}\label{e:reeb-deformed}
X_{\lambda_{f}}=e^{-\pi^{*}f}(\dth-\tl X_{f}).
\end{equation}
From this and \eqref{e:ham-grad-lift} we note at any point $p\in W$ where $f$ has a critical point, 
$X_{\lambda_{f}}(\theta, p)=e^{-f(p)}\dth$, and thus
the fiber in $S^{1}\times W$ over $p$ is a periodic orbit of the Reeb vector field with period
$2\pi e^{f(p)}$.

We are now ready to state the main theorem of the section,
which establishes a correspondence between gradient flows on a
symplectic manifold $(W, \omega=d\beta)$ and pseudoholomorphic
cylinders in the corresponding prequantization space
$(S^{1}\times W, d\theta+\pi^{*}\beta)$.
The idea of relating gradient flow lines of a Morse function to pseudoholomorphic cylinders in
a contact manifold originates in \cite{bourgeois}
(see also \cite{sft, behwz}).
In the present context this relationship can be seen as a generalization to the contact setting
of an idea of Floer from \cite{floer} (see also \cite{SZ, HS}).

\begin{theorem}\label{t:gradient-flow-hol-cylinders}
Let $(S^{1}\times W, \lambda=d\theta+\pi^{*}\beta)$ be a prequantization space
over an exact symplectic manifold $(W, \omega)$,
let $j\in\J(W, \omega=d\beta)$ be a compatible almost complex structure on $W$,
and let $J=\tl j\in\J(S^{1}\times W, \xi)$ be the corresponding
$S^{1}$-invariant compatible complex structure on $\xi=\ker\lambda$.
Given a smooth function $f:W\to\R$, consider smooth maps
$\gamma:\R\to W$,
$\theta:\R\to S^{1}$,
and
$a:\R\to\R$
satisfying the system of o.d.e.'s
\begin{align}
\dot\gamma(s)&=2\pi \nabla f(\gamma(s)) \label{e:ode-gamma} \\
\dot\theta(s)&=-2\pi \beta(\nabla f(\gamma(s))) \label{e:ode-theta}\\
\dot a(s)&=2\pi e^{f(\gamma(s))} \label{e:ode-a}
\end{align}
with $\nabla f$ denoting the gradient with respect to the metric
$g_{j}=\omega(\cdot, j\cdot)$.
Then the map $\tl u=(a, u):\R\times\S^{1}(\approx\R/\Z)\to \R\times S^{1}(\approx\R/2\pi\Z)\times W$ defined by
\[
\tl u(s, t)=(a(s, t), u(s, t))=(a(s), \theta(s)+2\pi t, \gamma(s))
\]
is a pseudoholomorphic cylinder for the data
$(e^{\pi^{*}f}\lambda, J)$ with Hofer energy\footnote{
We remark that finiteness of the energy here does not immediately
imply that the cylinders approach periodic orbits because $W$, being equipped with an
exact symplectic form, is necessarily noncompact.  Consider, for example, the symplectic manifold
$(\R^{2}, dx\wedge dy=d(x\,dy))$ and the function
$f(x, y)=\arctan x$.  The pseudoholomorphic cylinders in the appropriate prequantization space
covering gradient flow lines in the base have finite energy
as a result of \eqref{e:energy-formula}
since the function $f$ is bounded, but
the cylinders do not approach periodic orbits since
the function $f$ has no critical points.
}
\begin{equation}\label{e:energy-formula}
E(u)=2\pi \lim_{s\to\infty}e^{f(\gamma(s))}\in [0, +\infty].
\end{equation}
\end{theorem}

\begin{proof}
We compute using \eqref{e:horizontal-lift}
\begin{align*}
\tl u_{s}(s, t)
&=\dot a(s)\,\partial_{a}+\dot\theta(s)\,\dth+\dot\gamma(s) \\
&=2\pi e^{f(\gamma(s))}\,\partial_{a}-2\pi \beta(\nabla f(\gamma(s)))\,\dth+2\pi \nabla f(\gamma(s)) \\
&=2\pi \bp{e^{f(\gamma(s))}\,\partial_{a}+\wt{\nabla f}(u(s, t))}
\intertext{and similarly using \eqref{e:reeb-deformed}}
\tl u_{t}(s, t)
&=2\pi \dth \\
&=2\pi \bp{\dth-\wt{X_{f}}}(u(s, t))+2\pi \wt{X_{f}}(u(s, t)) \\
&=2\pi \bp{e^{f(\gamma(s))}X_{e^{f}\lambda}(u(s,t))+\wt{X_{f}}(u(s, t))}.
\end{align*}
It then follows from the definition \eqref{e:R-invariant-J} of the $\R$-invariant extension
of $J$ to an almost complex structure on $\R\times S^{1}\times W$ and from
\eqref{e:ham-grad-lift}, that $\tl u$ satisfies the pseudoholomorphic map equation
\eqref{e:j-hol}.

It remains to compute the Hofer energy.
To do that, we first compute
\begin{align*}
u^{*}\lambda
&=\lambda(u_{s})\,ds+\lambda(u_{t})\,dt \\
&=\lambda(\wt{\nabla f})\,ds+\lambda(2\pi \dth)\,dt \\
&=2\pi dt.
\end{align*}
We then consider a smooth, increasing function $\varphi:\R\to[0, 1]$ with
$\lim_{s\to\infty}\varphi(s)=1$ and $\lim_{s\to-\infty}\varphi(s)=0$, and compute
\begin{equation}\label{e:energy-computation}
\begin{aligned}
\int_{[s_{0}, s_{1}]\times\S^{1}}\tl u^{*}d(\varphi e^{\pi^{*}f}\lambda)
&=\int_{[s_{0}, s_{1}]\times\S^{1}}d(\varphi(a) e^{f(\gamma)}2\pi dt) \\
&=\bp{\int_{\br{s_{1}}\times\S^{1}}-\int_{\br{s_{0}}\times\S^{1}}}2\pi \varphi(a) e^{f(\gamma)}\,dt  \\
&=2\pi \bp{\varphi(a(s_{1}))e^{f(\gamma(s_{1}))}-\varphi(a(s_{0}))e^{f(\gamma(s_{0}))}}.
\end{aligned}
 \end{equation}
From \eqref{e:ode-gamma} and \eqref{e:ode-a} we know that the function $e^{f\circ\gamma}$ is increasing
and $a$ is strictly increasing with increasing derivative.
Thus $\lim_{s\to\infty}a(s)=+\infty$ and we can conclude that
\[
\lim_{s_{1}\to\infty}\varphi(a(s_{1}))e^{f(\gamma(s_{1}))}
=\bp{\lim_{s_{1}\to\infty}\varphi(a(s_{1}))}\bp{\lim_{s_{1}\to\infty}e^{f(\gamma(s_{1}))}}
=\varphi(+\infty)\lim_{s_{1}\to\infty}e^{f(\gamma(s_{1}))}
=\lim_{s_{1}\to\infty}e^{f(\gamma(s_{1}))}.
\]
Again using that $e^{f\circ\gamma}$ is increasing we can know that $\lim_{s_{0}\to-\infty}e^{f(\gamma(s_{0}))}$ exists and
is either positive or zero.  If the case that this limit is positive, we know from
\eqref{e:ode-a} that $\lim_{s_{0}\to-\infty}a(s_{0})=-\infty$ and hence that
$\lim_{s_{0}\to-\infty}\varphi(a(s_{0}))=0$.
In either case, we conclude that
\[
\lim_{s_{0}\to-\infty}\varphi(a(s_{0}))e^{f(\gamma(s_{0}))}
=0
\]
because it's a product of increasing, positive functions, at least one of which limits to $0$ as $s_{0}\to-\infty$.
Hence, taking limits in \eqref{e:energy-computation} above leads to
\[
\int_{\R\times\S^{1}}\tl u^{*}d(\varphi e^{f}\lambda)=2\pi \lim_{s\to\infty}e^{f(\gamma(s))}
\]
for any $\varphi\in\Xi$, which establishes $E(u)=2\pi \lim_{s\to\infty}e^{f(\gamma(s))}$ as claimed.
\end{proof}

\section{A gradient flow with a $1$-dimensional limit set}\label{s:gradient}

In this section we construct a function so that the
omega limit set of all of its nontrivial gradient flow lines is diffeomorphic to a circle.
This function will be used in the next section in conjunction with
Theorem \ref{t:gradient-flow-hol-cylinders} above
to construct finite-energy cylinders and planes localized near a transverse knot which have
tori as limit sets.

The main theorem of this section is the following.
\begin{theorem}\label{t:function-construction}
For any $\delta>0$, there exists a smooth function $F_{\delta}:\R\times S^{1}(\approx\R/2\pi\Z) \to \R$ so that
\begin{itemize}
\item $F_{\delta}(s, t)=s$ for $s\le\delta$,
\item $s\le F_{\delta}(s, t)< 0$ for $s\in(-\delta, 0)$,
\item $F_{\delta}(s, t)=0$ for $s\ge 0$,
\item $dF_{\delta}(s, t)\ne 0$ for $s<0$,
\end{itemize}
and so that for any choice of Riemannian metric
on $\R\times S^{1}$, the solution to the initial value problem
\[
\gamma(\tau)=\nabla F_{\delta}(\gamma(\tau)) \qquad \gamma(0)=(s_{0},t_{0}) \text{ with $s_{0}<0$}
\]
exists for all $\tau\ge 0$ and has the circle $\br{0}\times S^{1}$ as its omega limit set.
\end{theorem}

\begin{remark}
The fact that there exist gradient flows with omega limit sets consisting of more than a single point
has been known for some time and
a qualitative description of a function like the one we construct below is given in
\cite[pg.\ 261]{curry}.  In \cite[Example 3, pgs.\ 13-14]{palis-demelo} a function
in $\R^{2}$ is given for which
it can be shown that there is at least one gradient flow line whose omega limit set is a circle.

An interesting feature of the functions $F_{\delta}$ provided by our theorem is that 
every nontrivial flow line, independent of the metric,
has the circle $\br{0}\times S^{1}$ as its omega limit set.
Since the behavior of the functions $F_{\delta}$ is especially simple outside of a neighborhood of this limit
set, this allows for a good deal of flexibility in constructing functions on a given Riemannian manifold
whose gradients will have flow lines having a circle as an omega limit set.
For example, it is a straightforward corollary of this theorem that
one can construct smooth functions
on any Riemannian $2$-manifold $(M^2, g)$
having any desired embedded circle as the limit set of some gradient flow line.
Indeed, we can either identify a neighborhood of a given embedded circle with
$(-\ep, \ep)\times \R/2\pi\Z$
or, in the nonorientable case, we can identify a double cover of a neighborhood
of the circle
with $(-\ep, \ep)\times\R/2\pi\Z$ with the nontrivial deck-transformation of the cover
being given by the map $(s, t)\mapsto (-s, t+\pi)$.
We then consider the function $G(s, t):=F_{\ep/2}(s, t)+F_{\ep/2}(-s, t+\pi)$
which is invariant under the action of the deck transformation in the nonoriented case,
agrees with $-\abs{s}$ for $\abs{s}\in (\ep/2, \ep)$,
and
which will have $\br{0}\times S^{1}$ as the omega limit set of any gradient flow line 
starting in the neighborhood.
The function can then be extended to a smooth function on the entire surface using 
an appropriate cutoff function.
\end{remark}

We will construct the function in the following paragraph and prove
that it has the required properties in a series of lemmas.
Throughout this section we will make no notational distinction between smooth functions
with domain $\R\times S^{1}\approx\R\times\R/2\pi\Z$ and functions on
$\R^{2}$ which are $2\pi$-periodic in the second variable.

We consider the function $G:\R\times \R/2\pi\Z=\br{(s, t)}\to\R$ defined by\footnote{
As will become clear from our proof, the $5/4$ in our example can be replaced with any constant
strictly bigger than $1$ and strictly less than $\sqrt{2}$.
}
\[
G(s, t)=
\begin{cases}
e^{1/s}\bp{\sin(1/s+t)-5/4} & s<0 \\
0 & s\ge 0
\end{cases}
\]
and we note that $G$ is smooth and that\footnote{
The left-most part of this inequality can be seen from the following argument.
To show that $-\frac{9}{4}e^{1/s}-s$ is positive for all $s<0$, it suffices to show that
$g(t)=\frac{9}{4}te^{t}+1$ is positive for all $t<0$.  A straightforward argument using single-variable calculus
then shows that $g(-1)=-\frac{9}{4}e^{-1}+1>0$ is the absolute minimum of the function $g$ on $\R$.
}
\begin{equation}\label{e:bounds-G}
s< -\tfrac{9}{4} e^{1/s} \le G(s, t)\le -\tfrac{1}{4}e^{1/s} \text{ for all $s<0$.}
\end{equation}
For a given value $\delta>0$
we let $\eta:\R\to[0, 1]$ be a smooth cut-off function satisfying
\[
\eta(t)=
\begin{cases}
0 & s<-\delta \\
1& s>-\delta/2
\end{cases}
\]
and $\eta'(t)\ge 0$ everywhere and define the function
$F:\R\times\R/2\pi\Z\to\R$ by
\begin{equation}\label{e:F-definition}
F(s, t)=(1-\eta(s))s+\eta(s)G(s, t).
\end{equation}
We observe that this definition with \eqref{e:bounds-G} implies that
\begin{equation}\label{e:bounds-F}
s \le F(s, t)\le -\tfrac{1}{4}e^{1/s} \text{ for all $s<0$}
\end{equation}
so the first three properties required of $F_{\delta}$ in the theorem are clearly satisfied.
The fourth property, concerning the critical set, is then addressed by the following lemma.

\begin{lemma}\label{l:F-crit-set}
The set of critical points of the above defined function $F$ is $s\ge 0$.
\end{lemma}

\begin{proof}
We first compute for $s<0$
\begin{equation}\label{e:G-s}
\begin{aligned}
G_{s}(s, t)
&=-1/s^{2}e^{1/s}\bp{\sin(1/s+t)+\cos(1/s+t)-\tfrac{5}{4}} \\
&=-1/s^{2}e^{1/s}\bp{\sqrt{2}\sin(1/s+t+\tfrac{\pi}{4})-\tfrac{5}{4}} \\
\end{aligned}
\end{equation}
and
\begin{equation}\label{e:G-t}
G_{t}(s, t)
=e^{1/s}\cos(1/s+t)
\end{equation}
and note then that
\begin{align*}
dG(s^{2}\,\ds+\dt)
&=s^{2}G_{s}+G_{t} \\
&=-G
\end{align*}
which is everywhere positive for $s<0$.
We then compute
\begin{align*}
dF(s, t)
&=\eta'(s)(G(s, t)-s)\,ds+(1-\eta(s))\,ds+\eta(s)dG(s, t)
\end{align*}
and thus
\[
dF(s^{2}\,\ds+\dt)
=s^{2}\eta'(s)(G(s, t)-s)+(1-\eta(s))s^{2}+\eta(s)(-G),
\]
which we claim is always positive for $s<0$.  Indeed the first term is always nonnegative since,
as observed above in \eqref{e:bounds-G},
$G(s, t)\ge-\frac{9}{4}e^{1/s}>s$ for all $s<0$.
Meanwhile the second two terms are the convex sum of positive quantities and thus always positive.
We've thus found a vector field $v=s^{2}\,\ds+\dt$ for which $dF(v)>0$ for $s<0$ which shows that
$F$ has no critical points for $s<0$, and hence the critical set of $F$ is $s\ge 0$ where $F$ vanishes identically.
\end{proof}

As an immediate corollary we are able to show that the $\R$-component of any
nontrivial gradient flow line of $F$ converges to $0$ in forward time.

\begin{lemma}\label{l:F-forward-time}
For an arbitrary Riemannian metric $g$ on $\R\times S^{1}$ and a point $(s_{0}, t_{0})\in\R^{-}\times S^{1}$,
the solution $\gamma(\tau)=(s(\tau), t(\tau))\in\R\times S^{1}$ to
\begin{equation}\label{e:gradient-flow-ode}
\gamma'(\tau)=\nabla^{g} F(\gamma(\tau)) \qquad \gamma(0)=(s_{0}, t_{0})
\end{equation}
exists for all $\tau\ge 0$ and $\lim_{\tau\to\infty}s(\tau)=0$.
\end{lemma}

\begin{proof}
Since $dF(s, t)=0$ for $s\ge 0$ and $F(s, t)=s$ agrees with $s$ for $s< -\delta$, we know that
any solution to
\eqref{e:gradient-flow-ode} stays bounded in a set of the form
$[a, 0]\times S^{1}$ in forward time which implies that
the solution exists for all $\tau\ge 0$.
Given that the solution $\gamma(\tau)$ exists and is bounded in forward time,
we know from general properties of gradient flows
that $\lim_{\tau\to\infty}F(\gamma(\tau))$ exists and is equal to a critical value of $F$.
Since we have just seen in Lemma \ref{l:F-crit-set}, that $0$ is the unique critical value of $F$,
we conclude $\lim_{\tau\to\infty}F(\gamma(\tau))=0$.  This with \eqref{e:bounds-F} implies that
$\lim_{\tau\to\infty}s(\tau)=0$.
\end{proof}

The key step to proving the claim about the omega limit sets of flow lines of $F$ is the following lemma.

\begin{lemma}\label{l:z-bounded}
Let $\gamma(\tau)=(s(\tau), t(\tau))$ be a solution to \eqref{e:gradient-flow-ode}, and let
$\tl t:\R^{+}\to\R$ be a choice of lift of $t:\R^{+}\to S^{1}$.
Then the function $z:\R^{+}\to\R$ defined by
\[
z(\tau)=\frac{1}{s(\tau)}+\tl t(\tau)
\]
is bounded.
\end{lemma}

\begin{proof}
Let
\[
\begin{bmatrix}
A(s, t) & B(s, t) \\ B(s, t) & C(s, t)
\end{bmatrix}
\]
be the matrix of the dual metric to $g$ with respect to the coordinate basis
$\br{ds, dt}$ for $T^{*}(\R\times S^{1})$, and note positive definiteness tells us that 
$A(s, t)$ and $C(s, t)$ are positive for all $(s, t)\in\R\times S^{1}$.
Furthermore, since $\gamma(\tau)$ remains in a compact region for all $\tau\ge 0$, we can conclude
that the functions
$A(\tau):=A(s(\tau), t(\tau))$, $B(\tau):=B(s(\tau), t(\tau))$, and $C(\tau):=C(s(\tau), t(\tau))$
are bounded and that $A(\tau)$ and $C(\tau)$ are bounded away from zero
(or, equivalently that $A^{-1}(\tau)$ and $C^{-1}(\tau)$ are bounded).

From Lemma \ref{l:F-forward-time} and the definition \eqref{e:F-definition} of $F$ it follows
that $F(\gamma(\tau))=G(\gamma(\tau))$ for sufficiently large $\tau$.
For such values of $\tau$ we use \eqref{e:G-s}-\eqref{e:G-t}
with the boundedness of $A(\tau)$, $B(\tau)$, $C(\tau)$, and $A^{-1}(\tau)$
to compute
\begin{align*}
s'&=A(s, t) G_{s}(s, t)+B(s, t) G_{t}(s, t) \\
&=-s^{-2}e^{1/s}A(s, t)\bp{\sqrt{2}\sin(1/s+t+\tfrac{\pi}{4})-\tfrac{5}{4}+O(s^{2})}
\intertext{and}
\tl t'
&=t' \\
&=B(s, t)G_{s}(s, t)+C(s,t)G_{t}(s, t) \\
&=s^{-4}e^{1/s}A(s, t)O(s^{2})
\end{align*}
with $O(s^{2})$ denoting, as usual, a function $h$ for which $s^{-2}h(s, t)$ remains bounded on a
deleted neighborhood of $s=0$.
We then have that
\begin{align*}
z'
&=-s^{-2}s'+\tl t' \\
&=s^{-4}e^{1/s}A(s, t)\bp{\sqrt{2}\sin(1/s+ t+\tfrac{\pi}{4})-\tfrac{5}{4}+O(s^{2})} \\
&=s^{-4}e^{1/s}A(s, t)\bp{\sqrt{2}\sin(z+\tfrac{\pi}{4})-\tfrac{5}{4}+O(s^{2})}
\end{align*}
and so, for sufficiently large values of $\tau$ (and thus sufficiently small values of $s(\tau)$), we'll have that
\begin{equation}\label{e:z-inequality}
\sqrt{2}\sin(z(\tau)+\tfrac{\pi}{4})-\tfrac{11}{8}
\le
\bp{s(\tau)^{-4}e^{1/s(\tau)}A(\tau)}^{-1}z'(\tau)
\le
\sqrt{2}\sin(z(\tau)+\tfrac{\pi}{4})-\tfrac{9}{8}.
\end{equation}
We claim this lets us conclude that $z(\tau)$ is bounded.
Indeed, since $\frac{9}{8}\in(-\sqrt{2}, \sqrt{2})$,
the solution set to the inequality
\[
\sqrt{2}\sin(z+\tfrac{\pi}{4})-\tfrac{9}{8}<0
\]
is a countable union of intervals which is invariant under translation by $2\pi\Z$.
By \eqref{e:z-inequality}, $z(\tau)$ can't cross these intervals in the positive direction
once $\tau$ is sufficiently large for \eqref{e:z-inequality} to hold.
Similarly, since $\frac{11}{8}\in(-\sqrt{2}, \sqrt{2})$, $z(\tau)$ can't cross the intervals where
\[
\sqrt{2}\sin(z+\tfrac{\pi}{4})-\tfrac{11}{8}>0
\]
in the negative direction once $\tau$ is sufficiently large.
We conclude that $z(\tau)$ is bounded for $\tau\in[0, \infty)$.
\end{proof}

We now complete the proof of the main theorem of the section
\begin{proof}[Proof of Theorem \ref{t:function-construction}]
By construction and Lemma \ref{l:F-crit-set}, $F$ satisfies all required properties,
and it remains to show that the 
omega limit set of a solution to \eqref{e:gradient-flow-ode} is the circle $\br{0}\times S^{1}$.
Let $\gamma(\tau)=(s(\tau), t(\tau))$ be a solution to \eqref{e:gradient-flow-ode} and let
$\tl t:\R^{+}\to\R$ be a lift of $t:\R^{+}\to S^{1}$.
We have shown in Lemma \ref{l:z-bounded} above that the function
$z=1/s+\tl t$ is bounded on $[0, \infty)$.  Since we know from Lemma \ref{l:F-forward-time} that
$\lim_{\tau\to\infty}s(\tau)=0$ and since $s(\tau)<0$ for all $\tau\ge 0$,
we can conclude that $\lim_{\tau\to\infty}\frac{1}{s(\tau)}=-\infty$.
This in turn lets us conclude
$\tl t(\tau)$ approaches $+\infty$ as $\tau\to\infty$ or else $z$ would not be bounded.
By continuity, the equation $\tl t(\tau)=c$ has a solution $\tau_{c}$ for all $c\ge \tl t(0)$.
We then conclude that for any $\tau_{0}\in\R$ and any $t_{0}\in S^{1}$, there
exists a $\tau_{t_{0}}>\tau_{0}$ so that $t(\tau_{t_{0}})=t_{0}$.
This with the fact that $s(\tau)\to 0$ as $\tau\to\infty$ shows that
$\br{0}\times S^{1}$ is the omega limit set of $\gamma$.
\end{proof}

\section{Finite-energy cylinders and planes with tori as limit sets}\label{s:main-proof}
Here we prove our main theorem, Theorem \ref{t:main-theorem}, that is,
we construct examples of finite-energy cylinders and finite-energy planes
having tori as limit sets.  The constructions take place in an arbitrarily small
tubular neighborhood of a standard model of a transverse knot,
so we begin by recalling some basic facts about transverse knots
and explaining why this construction suffices to prove the main theorem.

Let $(M^{2n+1}, \xi=\ker\lambda)$ be a contact manifold.
An embedding $\gamma:S^{1}\to M$ is said to be a transverse knot if $\gamma$
is everywhere transverse to $\xi$ or equivalently, if $\lambda(\dot\gamma)$ is never zero.
Transverse knots exist in abundance in any contact manifold.  Indeed, by the well-known Darboux theorem
for contact structures, there exists a contactomorphism  ---
that is a diffeomorphism preserving the contact structure
---
between a neighborhood of any point in a contact manifold $(M^{2n+1}, \xi)$ and a neighborhood
of $0$ in $\R^{2n+1}=\br{(z, x_{i}, y_{i})}$ equipped with the contact structure
$\xi_{0}=\ker\lambda_{0}$ where $\lambda_{0}$ is the contact form
\[
\lambda_{0}=dz+\alpha_{n}
\]
with
\begin{equation}\label{e:alpha-defn}
\alpha_{n}=\sum_{i=1}^{n}x_{i}\,dy_{i}-y_{i}\,dx_{i}
\end{equation}
(see e.g.\ \cite[Theorem 2.24]{geiges}).
Since, for a given $k\in\Z\cap[1, n]$
and any constants $r>0$, $c_{i}$, $d_{i}\in\R$, circles of the form
\[
\text{$x_{k}^{2}+y_{k}^{2}=r^{2}$, and $x_{i}=c_{i}$, $y_{i}=d_{i}$, for $i\ne k$}
\]
are easily seen to be transverse to the contact structure,
we can conclude that transverse knots exist in every contact manifold and, indeed,
that transverse knots exist in any neighborhood of a given point in a contact manifold.

We next recall that one can use a Moser argument to prove a neighborhood theorem for transverse knots 
which tells us that there exists a contactomorphism between some neighborhood of any given
transverse knot and a neighborhood of $S^{1}\times\br{0}$ in
$S^{1}\times\R^{2n}=\br{(\theta, x_{i}, y_{i})}$
equipped with the contact structure
$\xi_{0}=\ker \lambda_{0}$
where
\begin{equation}\label{e:transverse-standard}
\lambda_{0}=d\theta+\alpha_{n}
\end{equation}
with $\alpha_{n}$ as defined in \eqref{e:alpha-defn} above
(see e.g.\ \cite[Theorem 2.32/Example 2.33]{geiges}).
We will refer to
$S^{1}\times\br{0}\subset (S^{1}\times\R^{2n}, \xi_{0})$
as the standard model of a transverse knot in $S^{1}\times\R^{2n}$.

Given the facts recalled in the previous two paragraphs, 
it suffices for the proof of our main theorem
to construct the desired finite-energy planes and cylinders
in any given neighborhood of the standard model of a transverse knot in $S^{1}\times\R^{2n}$.
Since $d\alpha_{n}=2\sum_{i=1}^{n}dx_{i}\wedge dy_{i}$ is a symplectic form on
$\R^{2n}$, $S^{1}\times\R^{2n}$ equipped with the contact form
\eqref{e:transverse-standard} has the structure of a prequantization space.
We can thus apply Theorems \ref{t:gradient-flow-hol-cylinders} and \ref{t:function-construction}
to construct a finite-energy cylinder having tori of periodic orbits as its limit sets.

\begin{theorem}
Let $r_{+}>r_{-}>0$.  Then there exists a smooth function
$F:S^{1}\times\R^{2n}$ and an almost complex structure
$J\in\J(S^{1}\times\R^{2n}, \xi_{0})$ so that passing through every point
$(\theta_{0}, p, z)\in S^{1}\times \R^{2(n-1)}\times \R^{2}$ with $\abs{z}\in (r_{-}, r_{+})$
is a finite-energy cylinder  for the data $(e^{F}\lambda, J)$
with limit sets equal to the union of tori
$S^{1}\times\br{p}\times\br{\abs{z}=r_{-}}\cup\br{\abs{z}=r_{+}}
\in S^{1}\times\R^{2(n-1)}\times\R^{2}$.
\end{theorem}

\begin{proof}
With $F_{\delta}$ a function with the properties stated in Theorem \ref{t:function-construction},
we consider a function $G:\R\times S^{1}\to\R$ defined by
\[
G(\rho, \phi)=
\begin{cases}
F_{1/4}(\rho, \phi) & \rho>-3/4 \\
-F_{1/4}(-\rho-1, \phi)-1 & \rho<-1/4
\end{cases}
\]
which defines a smooth function since $F_{1/4}(\rho, \phi)=\rho=-F_{1/4}(-\rho-1, \phi)-1$ for
$\rho\in [-3/4, -1/4]$.
For any initial condition $(\rho_{0}, \phi_{0})\in [-3/4, -1/4]\times S^{1}$ the
forward gradient flow of $G$ for any metric
agrees with that of $F_{1/4}$ and thus limits to $\br{0}\times S^{1}$.
Similarly, for any initial condition $(\rho_{0}, \phi_{0})\in [-3/4, -1/4]\times S^{1}$ the backward gradient flow
of $G$ for any metric
agrees with that of
$-F_{1/4}(-\rho-1, \phi)-1$
which is conjugated to the forward gradient flow of $F_{1/4}$ by reflection and translation
and thus limits in backward time to
$\br{-1}\times S^{1}$.

We consider the diffeomorphism 
$p:\R\times\R/2\pi\Z\to\R^{2}\setminus\br{0}$
defined by
\[
p(\rho, \phi)=(r_{+}(r_{+}/r_{-})^{\rho}\cos \phi, r_{+}(r_{+}/r_{-})^{\rho}\sin \phi)
\]
which maps the circles
$\br{-1}\times S^{1}$ and $\br{0}\times S^{1}$ to
the circles $\abs{z}=r_{-}$ and $\abs{z}=r_{+}$ respectively. 
We then define a function $F:\R^{2n}\to\R$ by
\[
F(x_{1}, y_{1}, \dots, x_{n}, y_{n})=
\begin{cases}
G(p^{-1}(x_{n}, y_{n})) & (x_{n}, y_{n})\ne 0 \\
-1 & (x_{n}, y_{n})= 0
\end{cases}
\]
which defines a smooth function since $G(\rho, \phi)=-F_{1/4}(-\rho-1, \phi)-1$ for $\rho\le -1$ and
hence $G(p^{-1}(z))=-1$ for $\abs{z}<r_{-}$.
We observe that for any metric on $\R^{2n}\approx\R^{2(n-1)}\times\R^{2}$
for which the last $T\R^{2}$ is everywhere orthogonal to $T\R^{2(n-1)}$
we will have that
$\nabla F=(0, \nabla G)\in \R^{2(n-1)}\times\R^{2}$
and thus the gradient flow of $F$ for initial points
$(p, z)\in\R^{2(n-1)}\times\R^{2}$ with $\abs{z}\in(r_{-}, r_{+})$
will have the circles $\abs{z}\in\br{r_{-}, r_{+}}$ as limit sets.

Choosing then an almost complex structure
$J\in\J(\R^{2n}, d\alpha_{n})$
on $\R^{2(n-1)}\times \R^{2}$ which preserves the two factors
(for example the standard $J_{0}$ defined by $J_{0}\partial_{x_{i}}=\partial_{y_{i}}$),
we know from Theorem \ref{t:gradient-flow-hol-cylinders} that gradient flow lines
of $F$ on $\R^{2n}$ with respect to the metric $d\beta(\cdot, J\cdot)$ lift to
finite-energy cylinders in $\R\times S^{1}\times\R^{2n}$ for the data
$(e^{\pi^{*}F}\lambda, J)$.
Since the nonconstant gradient flow lines for the function $F$ will have
the circles $\abs{z}=r_{\pm}$ as limit sets, the corresponding finite-energy cylinders in
$S^{1}\times\R^{2n}$ will have the tori $S^{1}\times\br{p}\times\br{\abs{z}=r_{\pm}}$ as limit
sets.
\end{proof}

Using Theorem \ref{t:gradient-flow-hol-cylinders} to construct finite-energy a plane
with a torus as a limit set is somewhat more subtle since the theorem only
tells us how to construct a cylinder from a gradient flow line.
To construct a plane we will use the theorem to construct a cylinder with a removable singularity.

\begin{theorem}\label{t:plane-construction}
Let $r_{0}>0$.  Then there exists a smooth function
$\tl F:S^{1}\times\R^{2n}$ and an almost complex structure
$J\in\J(S^{1}\times\R^{2n}, \xi_{0})$ so that passing through every point
$(\theta_{0}, 0, z)\in S^{1}\times \R^{2(n-1)}\times \R^{2}$ with $\abs{z}<r_{0}$
is a finite-energy plane  for the data $(e^{\tl F}\lambda_{0}, J)$
with limit set equal to the embedded torus
$S^{1}\times\br{0}\times\br{\abs{z}=r_{0}}\in S^{1}\times\R^{2(n-1)}\times\R^{2}$.
\end{theorem}

The strategy of the proof is to consider a set which is contactomorphic to the complement of the
$x_{n}=y_{n}=0$ locus of a standard model of a transverse knot and show that this can be given the
structure of a prequantization space with respect to the angular variable on
$\br{(x_{n}, y_{n})}\setminus\br{0}$.  We then use Theorems \ref{t:gradient-flow-hol-cylinders}
and \ref{t:function-construction} to construct a pseudoholomorphic cylinder which has a removable
puncture mapped to the the $x_{n}=y_{n}=0$ locus.

We begin with a computational lemma.
\begin{lemma}\label{l:plane-construction}
Consider $W:=S^{1}\times\R^{2(n-1)}\times\R=\br{(\theta, x_{i}, y_{i}, \rho)}$ equipped with the
$1$-form
\begin{equation}\label{e:beta-planes}
\beta=e^{-2\rho}\bp{d\theta+\alpha_{n-1}}
\end{equation}
with $\alpha_{n-1}$ as defined in \eqref{e:alpha-defn}.  Then:
\begin{itemize}
\item $d\beta$ is a symplectic form on $W$.
\item Consider the corresponding prequantization space $(S^{1}\times W, \lambda:=d\phi+\pi^{*}\beta)$
over $W$.  With $\lambda_{0}$ as defined in \eqref{e:transverse-standard},
the map
\[
\Phi:(S^{1}\times W, \xi=\ker\lambda)\to
(S^{1}\times \R^{2(n-1)}\times(\R^{2}\setminus\br{0}), \xi_{0}=\ker\lambda_{0})
\]
defined by
\begin{equation}\label{e:Phi-definition}
\Phi(\phi, \theta, x_{i}, y_{i}, \rho)=(\theta, x_{i}, y_{i}, e^{\rho}\cos\phi, e^{\rho}\sin\phi)
\end{equation}
is a contactomorphism and, in particular, 
\begin{equation}\label{e:lambda-0-pullback}
\Phi^{*}\lambda_{0}=e^{2\rho}\lambda.
\end{equation}
\item For any choice of $j_{0}\in\J(\R^{2(n-1)}, d\alpha_{n-1})$ the endomorphism
$j_{1}\in\End(TW)$ defined by
\begin{equation}\label{e:j1-definition}
\begin{gathered}
j_{1}(\theta, p, \rho)\partial_{\rho}=-e^{2\rho}\dth,
\qquad
j_{1}(\theta, p, \rho)\dth=e^{-2\rho}\partial_{\rho}, \text{ and} \\
j_{1}(\theta, p, \rho)v=j_{0}(p)v-\alpha_{n-1}(j_{0}(p)v)\,\dth+\alpha_{n-1}(v)e^{-2\rho}\,\partial_{\rho} \text{ for $v\in T\R^{2(n-1)}$}.
\end{gathered}
\end{equation}
is an almost complex structure on $W$ compatible with $d\beta$, i.e.\ $j_{1}\in\J(W, d\beta)$, and
the corresponding metric $g_{j_{1}}:=d\beta\circ(I\times j_{1})$ on $W$ is given by
\begin{equation}\label{e:metric-plane}
g_{j_{1}}
= 2 \,d\rho\otimes d\rho
+2e^{-4\rho}(d\theta+\alpha_{n-1})\otimes\bp{d\theta+\alpha_{n-1}}
+e^{-2\rho}d\alpha_{n-1}\circ (I\times j_{0}).
\end{equation}
\item Let $\tl j_{1}\in\J(S^{1}\times W, \xi)$ be the $S^{1}$-invariant complex structure on
$\xi$ determined by $j_{1}$ as defined above, i.e.\ $\tl j$ is the complex structure characterized by
$\wt{j_{1}v}=\tl j_{1}\tl v$ with
\begin{equation}\label{e:lift-planes-proof}
\tl v=-\beta(v)\partial_{\phi}+v=-e^{-2\rho}\bp{d\theta(v)+\alpha_{n-1}(v)}\partial_{\phi}+v
\end{equation}
the lift of $v$ to an $S^{1}$-invariant section of $\xi$ from \eqref{e:horizontal-lift}.
Then $\Phi_{*}\tl j_{1}=d\Phi\circ \tl j_{1}\circ d\Phi^{-1}\in\J(S^{1}\times \R^{2(n-1)}\times\R^{2}\setminus\br{0}, \xi_{0})$ admits a smooth extension to a compatible
$J\in\J(S^{1}\times\R^{2n}, \xi_{0})$.
\end{itemize}
\end{lemma}

Assuming for the moment the results of the lemma, we proceed with the proof of Theorem \ref{t:plane-construction}.

\begin{proof}[Proof of Theorem \ref{t:plane-construction}]
Given $r_{0}>0$ we define a smooth function
$G:W\to\R$ by
\[
G(\theta, p, \rho)=2(F_{1}(\rho-\log r_{0}, \theta)+\log r_{0})
\]
with $F_{1}$ a function satisfying the properties given in
Theorem \ref{t:function-construction} with $\delta=1$.
We note that as a result of the definition and of Theorem \ref{t:function-construction},
$G(\theta, p, \rho)=2\rho$ for $\rho<\log r_{0}-1$ and
$G(\theta, p, \rho)=2\log r_{0}$ for $\rho\ge\log r_{0}$.
Moreover, since $\alpha_{n-1}=0$ along the $p=(x_{1}, y_{1}, \dots, x_{n-1}, y_{n-1})=0$ locus,
we have that
\begin{equation}\label{e:G-gradient}
\nabla G(\theta, 0, \rho)=\frac{1}{2}\bp{G_{\rho}\partial_{\rho}+G_{\theta}e^{4\rho}\dth}
=\drh F_{1}(\rho-\log r_{0}, \theta)\partial_{\rho}+\dth F_{1}(\rho-\log r_{0}, \theta)e^{4\rho}\dth
\end{equation}
where $\nabla G$ is the gradient with respect to the metric \eqref{e:metric-plane} on $W$.
Therefore, for any initial point $w_{0}=(\theta_{0}, 0, \rho_{0})\in S^{1}\times\R^{2(n-1)}\times\R$
with $\rho_{0}<\log r_{0}$,
the
solution $\gamma(s)$ to the equation
\begin{equation}\label{e:ode-gamma-G}
\dot\gamma(s)=2\pi\nabla G(\gamma(s))
\end{equation}
stays within the embedded cylinder
$S^{1}\times\br{0}\times\R\subset S^{1}\times\R^{2(n-1)}\times\R$
and agrees with a gradient flow for the function
$(\rho, \theta)\in \R\times S^{1}\mapsto 2(F_{1}(\rho-\log r_{0}, \theta)+\log r_{0})$ for an
appropriate metric on the cylinder.
Thus the flow exists in forward time and has the circle
$S^{1}\times\br{0}\times\br{\log r_{0}}\in S^{1}\times\R^{2(n-1)}\times\R$ as its omega limit set.
Meanwhile, using the fact that 
$F_{1}(\rho-\log r_{0}, \theta)=\rho-\log r_{0}$ for $\rho<\log r_{0}-1$,
we have from \eqref{e:G-gradient} that
\[
\nabla G(\theta, 0, \rho)=\partial_{\rho} \qquad \text{ for $\rho<\log r_{0}-1$}
\]
and thus that the solution $\gamma$ to \eqref{e:ode-gamma-G} is given by
$\gamma(s)=(\theta_{1}, 0, 2\pi s+s_{1})$
for sufficiently small $s$ with
$\theta_{1}\in S^{1}$ and $s_{1}\in\R$ appropriate constants.
Thus the flow exists
indefinitely in backward time as well.
Applying Theorem \ref{t:gradient-flow-hol-cylinders} we know that the map
$\tl u(s, t)=(a(s),\phi(s)+2\pi t, \gamma(s))\in\R\times S^{1}\times W$
where $a:\R\to\R$ and $\phi:\R\to S^{1}$ satisfy
$\dot a(s)=2\pi e^{G(\gamma(s))}$ and $\dot\phi(s)=2\pi \beta(\nabla G(\gamma(s)))$
is a finite-energy cylinder
for the data $(e^{\pi^{*}_{W}G}\lambda, \tl j_{1})$
with energy
$2\pi \lim_{s\to\infty}e^{G(\gamma(s))}=2\pi e^{2\log r_{0}}=2\pi  r_{0}^{2}$ and the torus
$S^{1}\times S^{1}\times \br{0}\times\br{\log r_{0}}$ as its limit set.
Moreover, since $G(\theta, p, \rho)=2\rho$ and $\nabla G(\theta, p, \rho)=\partial_{\rho}$
for $\rho<\log r_{0}-1$, 
there exist constants $a_{1}\in\R$, $\theta_{1}\in S^{1}$, $s_{1}\in\R$, and $t_{1}\in S^{1}$ so that
\begin{equation}\label{e:map-near-zero}
\tl u(s, t)
=(\pi e^{4\pi s}+a_{1},t_{1}+2\pi t, \theta_{1},0, 2\pi s+s_{1}) \in\R\times S^{1}\times S^{1}\times\R^{2(n-1)}\times\R
=\R\times S^{1} \times W.
\end{equation}
for sufficiently negative $s$.

We next show that the map
$\tl v:=(a, \Phi\circ u):\R\times\S^{1}\to\R\times S^{1}\times\R^{2n}$,
with $\Phi:(S^{1}\times W, \xi) \to (S^{1}\times \R^{2n}, \xi_{0})$ the contactomorphism
defined in \eqref{e:Phi-definition},
has a removable singularity at $-\infty$ and thus extends to a pseudoholomorphic plane.
We first note that \eqref{e:lambda-0-pullback} gives us
\[
[\Phi^{-1}]^{*}\bp{e^{(\pi_{W}^{*}G)}\lambda}=e^{(\pi_{W}^{*}G-2\rho)\circ\Phi^{-1}}\lambda_{0}
\]
and, since $G(\theta, p, \rho)=2\rho$ for $\rho<\log r_{0}-1$, we'll have that
$\tl F:=(\pi_{W}^{*}G-2\rho)\circ\Phi^{-1}$ extends to a smooth function on
$S^{1}\times\R^{2n}$
and thus $[\Phi^{-1}]^{*}\bp{e^{(\pi_{W}^{*}G)}\lambda}=e^{\tl F}\lambda_{0}$
defines a contact form on $S^{1}\times\R^{2n}$.
Since, by Lemma \ref{l:plane-construction}, the pushed-forward complex
structure $\Phi_{*}\tl j_{1}=d\Phi\circ j_{1}\circ d\Phi^{-1}$ has a smooth extension to a
$J\in\J(S^{1}\times\R^{2n}, \xi_{0})$, it suffices to show that the map
$\tl v=(a, \Phi\circ u)$ has a smooth extension.  To see this,
we use the definition \eqref{e:Phi-definition} with \eqref{e:map-near-zero} to compute
that
\[
\tl v=(a, \Phi\circ u)=(\pi e^{4\pi s}+a_{1},
\theta_{0},0, e^{2\pi s+s_{1}}\cos(t_{1}+2\pi t), e^{2\pi s+s_{1}}\sin(t_{1}+2\pi t))\in
\R\times S^{1}\times \R^{2(n-1)}\times\R^{2}
\]
for sufficiently negative $s$.
Precomposing with the biholomorphic map
$\psi:\C\setminus\br{0}\to\R\times S^{1}=\C/i\Z$ defined by
\[
\psi(z)=(\log\abs{z}/2\pi, \arg{z}/2\pi)
\]
we find that
\[
\tl v(\psi(z))=(\pi\abs{z}^{2}+a_{1}, \theta_{0}, 0, e^{s_{1}+it_{1}}z)
\in
\R\times S^{1}\times \R^{2(n-1)}\times\R^{2}(\approx \C)
\]
which clearly extends smoothly over $z=0$.
We note moreover that the limit set
$S^{1}\times S^{1}\times \br{0}\times\br{\log r_{0}}\subset S^{1}\times S^{1}\times\R^{2(n-1)}\times\R$
gets mapped by $\Phi$ to the embedded torus
$S^{1}\times\br{0}\times\br{\abs{z}= r_{0}}\in S^{1}\times\R^{2(n-1)}\times\R^{2}$,
while the set of points
$(\phi, \theta, 0, \rho)\in S^{1}\times S^{1}\times \R^{2(n-1)}\times\R$
with $\rho<\log r_{0}$ 
gets mapped by $\Phi$ to the set
$(\theta, 0, z)\in S^{1}\times\R^{2(n-1)}\times\R^{2}$ with $\abs{z}\in(0, r_{0})$.
Thus, since we were able, by appropriate choice of initial point of the flow of
$\nabla G$, to construct a pseudoholomorphic cylinder for the data
$(e^{\pi_{W}G}\lambda, \tl j_{1})$
through any point $(\phi, \theta, 0, \rho)\in S^{1}\times S^{1}\times \R^{2(n-1)}\times\R$,
we can construct a pseudoholomorphic plane for the data
$(e^{\tl F}\lambda_{0}, J)$
through any point
$(\theta, 0, z)\in S^{1}\times\R^{2(n-1)}\times\R^{2}$ with $\abs{z}<r_{0}$
as desired.
This completes the proof.
\end{proof}

\begin{remark}
If we choose an initial point in the proof of theorem to be a point
$(\theta_{0}, p_{0}, \rho_{0})\in S^{1}\times\R^{2(n-1)}\times\R$ with $p_{0}\ne 0$, one can still
construct a finite-energy plane from the resulting flow line with a limit set consisting of more than a single orbit,
although the limit set may be more complicated than a torus.
Indeed,
for a given choice of $j_{0}\in \J(\R^{2(n-1)}, d\alpha_{n-1})$, we let
$g_{j_{0}}=d\alpha_{n-1}(\cdot, j_{0})$ denote the associated metric and observe that
\[
\alpha_{n-1}(p)(v)=\frac{1}{2}d\alpha_{n-1}(p, v)=-\frac{1}{2}d\alpha_{n-1}(v, p)=\frac{1}{2}g_{j_{0}}(v, j_{0}p).
\]
Using this, one can compute the gradient of the function $G$ with respect to the metric
\eqref{e:metric-plane} to be given by
\begin{align*}
\nabla G(\theta, p, \rho)
&=2^{-1}\bp{
G_{\rho}\,\partial_{\rho}+e^{2\rho}G_{\theta}\bp{e^{2\rho}+\abs{p}_{j_{0}}}\,\dth
- e^{2\rho}G_{\theta}j_{0}p
} \\
&=
\drh F_{1}(\rho-\log r_{0}, \theta)\,\partial_{\rho}
+\dth F_{1}(\rho-\log r_{0}, \theta)e^{2\rho}\bp{e^{2\rho}+\abs{p}_{j_{0}}}\,\dth
-\dth F_{1}(\rho-\log r_{0}, \theta) e^{2\rho} j_{0}p.
\end{align*}
with $\abs{\cdot}_{j_{0}}$ the norm with respect to the metric $g_{j_{0}}$.
We note the the $\R^{2(n-1)}$-component of $\nabla G(\theta, p, \rho)$ is always orthogonal
to $p$.  If the almost complex structure $j_{0}$ is constant, we can thus conclude that
$\abs{p}_{j_{0}}$ is constant along the flow.
Thus, $\rho$- and $\theta$-components of the gradient flow for $G$ agree with a gradient flow for a
shift of the function $F$ on $\R\times S^1$ for an appropriate metric
(specifically the metric
$g=d\rho\otimes d\rho+e^{-2\rho}\bp{e^{2\rho}+c^{2}}^{-1}d\theta\otimes d\theta$ with
$c^{2}$ equal to the constant value of $\abs{p}_{j_{0}}$ along the flow line).
From this one can argue that under the projection
$S^{1}\times\R^{2(n-1)}\times\R^{2}\to S^{1}\times\br{0}\times\R^{2}$
the limit set of any plane obtained as a lift of a gradient flow of the function $G$ in our theorem
will project to a torus.
\end{remark}

Finally, to complete the proof of Theorem \ref{t:plane-construction},
we give the proof of Lemma \ref{l:plane-construction} above.

\begin{proof}[Proof of Lemma \ref{l:plane-construction}]
We first show that $d\beta$ is a symplectic form.  Computing, we have that
\begin{equation}\label{e:d-beta}
d\beta
=e^{-2\rho}(-2\,d\rho\wedge d\theta-2\,d\rho\wedge\alpha_{n-1}+d\alpha_{n-1})
\end{equation}
and hence
\begin{align*}
d\beta^{n}
&=e^{-2n\rho}(-2\,d\rho\wedge d\theta-2\,d\rho\wedge\alpha_{n-1})\wedge(d\alpha_{n-1})^{n-1} \\
&=-2e^{-2n\rho}\,d\rho\wedge d\theta\wedge(d\alpha_{n-1})^{n-1}
\end{align*}
which is nowhere vanishing on $W=S^{1}\times \R^{2(n-1)}\times\R$.
Hence $d\beta$ is a symplectic form on $W$ as claimed.

Next, we show that the map
$\Phi:(W, \xi=\ker\lambda)
\to (S^{1}\times\R^{2(n-1)}\times(\R^{2}\setminus\br{0}), \xi_{0}=\ker\lambda_{0})$
defined in \eqref{e:Phi-definition}
is a contactomorphism satisfying \eqref{e:lambda-0-pullback}.
From the definition \eqref{e:Phi-definition} of the map,
it's clear that $\Phi$ is a diffeomorphism and that
\[
\Phi^{*}d\theta=d\theta \qquad \Phi^{*}dx_{i}=dx_{i} \qquad \Phi^{*}dy_{i}=dy_{i}
\]
for $i\in\Z\cap[1, n-1]$, while a straightforward computation shows that
\begin{equation}\label{e:Phi-pullback}
\Phi^{*}(x_{n}\,dy_{n}-y_{n}\,dx_{n})=e^{2\rho}\,d\phi \qquad \Phi^{*}(x_{n}\,dx_{n}+y_{n}\,dy_{n})=e^{2\rho}\,d\rho.
\end{equation}
Computing then gives
\begin{align*}
\Phi^{*}\lambda_{0}
&=d\theta+\alpha_{n-1}+\Phi^{*}(x_{n}\,dy_{n}-y_{n}\,dx_{n}) \\
&=d\theta+\alpha_{n-1}+e^{2\rho}\,d\phi \\
&=e^{2\rho}\lambda
\end{align*}
which shows that $\Phi$ is a contactomorphism and establishes
\eqref{e:lambda-0-pullback} as claimed.

We next address the third point.
The fact that $j_{1}^{2}\partial_{\rho}=-\partial_{\rho}$ and $j_{1}^{2}\dth=-\dth$
is immediate from the definition \eqref{e:j1-definition}.  Meanwhile for $v\in T(\R^{2(n-1)})$, we use 
\eqref{e:j1-definition} twice with $j_{0}^{2}v=-v$ to compute
\begin{align*}
j_{1}^{2}v
&=j_{1}\bp{j_{0}v-\alpha_{n-1}(j_{0}v)\,\dth+\alpha_{n-1}(v)e^{-2\rho}\,\partial_{\rho}} \\
&=j_{1}(j_{0}v)-\alpha_{n-1}(j_{0}v)j_{1}\dth+\alpha_{n-1}(v)e^{-2\rho}j_{1}\partial_{\rho} \\
&=j_{0}^{2}v-\alpha_{n-1}(j_{0}^{2}v)\,\dth+\alpha_{n-1}(j_{0}v)e^{-2\rho}\,\partial_{\rho} \\
&\hskip.25in -\alpha_{n-1}(j_{0}v)(e^{-2\rho}\,\drh)+\alpha_{n-1}(v)e^{-2\rho}(-e^{2\rho}\dth) \\
&=-v
\end{align*}
which shows that $j_{1}$ is an almost complex structure on $W$.  To check compatibility of
$j_{1}$ with $d\beta$ we compute from \eqref{e:j1-definition} that
\begin{align*}
d\rho\circ j_{1}
&=e^{-2\rho}\bp{d\theta+\alpha_{n-1}} \\
d\theta\circ j_{1}
&=-e^{2\rho}\,d\rho-\alpha_{n-1}\circ j_{0}\circ d\pi_{\R^{2(n-1)}}. \\
dx_{i}\circ j_{1}&=dx_{i}\circ j_{0}\circ d\pi_{\R^{2(n-1)}} \\
dy_{i}\circ j_{1}&=dy_{i}\circ j_{0}\circ d\pi_{\R^{2(n-1)}}
\end{align*}
which with \eqref{e:d-beta} gives us
\[
d\beta\circ(I\times j_{1})
= 2 \,d\rho\otimes d\rho
+2e^{-4\rho}(d\theta+\alpha_{n-1})\otimes\bp{d\theta+\alpha_{n-1}}
+e^{-2\rho}d\alpha_{n-1}\circ (I\times j_{0}).
\]
as claimed.
By the assumption that $j_{0}$ is compatible with $d\alpha_{n-1}$, this is clearly
symmetric and positive definite, and thus $j_{1}\in\J(W, d\beta)$ as claimed.

Finally, we show that $\Phi_{*}\tl j_{1}$ has a smooth extension to
a compatible complex structure $J\in\J(S^{1}\times\R^{2n}, \xi_{0})$.
The contact structure $\xi_{0}=\ker\lambda_{0}$ is spanned by the smooth sections
\[
-\alpha_{n}(\partial_{x_{i}})\dth+\partial_{x_{i}} \qquad -\alpha_{n}(\partial_{y_{i}})\dth+\partial_{y_{i}}
\]
so it suffices to check that
$\Phi_{*}\tl j_{1}$ times each of these sections has a smooth continuation.
We first observe that from the definition \eqref{e:Phi-definition} of $\Phi$ we immediately have
\begin{equation}\label{e:Phi-pushforward-1}
\Phi_{*}\dth=\dth \qquad \Phi_{*}\partial_{x_{i}}=\partial_{x_{i}} \qquad  \Phi_{*}\partial_{y_{i}}=\partial_{y_{i}}
\end{equation}
for $i$ between $1$ and $n-1$, while \eqref{e:Phi-pullback} give us
\begin{equation}\label{e:Phi-pushforward-2}
\Phi_{*}\partial_{\rho}=x_{n}\,\partial_{x_{n}}+y_{n}\,\partial_{y_{n}}
\qquad
\Phi_{*}\partial_{\phi}=x_{n}\,\partial_{y_{n}}-y_{n}\,\partial_{x_{n}}.
\end{equation}
Thus,
for $v\in T(\R^{2(n-1)})=\operatorname{span}\br{\partial_{x_{i}}, \partial_{y_{i}}}_{i=1}^{n-1}$,
a straightforward computation using that $d\theta(v)=0$ along with \eqref{e:lift-planes-proof}
and \eqref{e:Phi-pushforward-1}
shows that
\[
-\alpha_{n}(v)\dth+v=\Phi_{*}(-\alpha_{n-1}(v)\wt{\dth}+\tl v).
\]
Computing further with this,
the definition \eqref{e:j1-definition} of $j_{1}$, and $\tl j_{1}\tl v=\wt{j_{1}v}$
then shows that
\begin{align*}
(\Phi_{*}\tl j_{1})(-\alpha_{n}(v)\dth+v)
&=(\Phi_{*}\tl j_{1})\Phi_{*}(-\alpha_{n-1}(v)\wt{\dth}+\tl v) \\
&=\Phi_{*}( -\alpha_{n-1}(v)\wt{j_{1}\dth}+\wt{j_{1}v}) \\
&=\Phi_{*}(-e^{-2\rho}\alpha_{n-1}(v)\wt{\drh}+\wt{j_{0}v}-\alpha_{n-1}(j_{0}v)\wt{\dth}+\alpha_{n-1}(v)e^{-2\rho}\wt{\drh}) \\
&=\Phi_{*}(-\alpha_{n-1}(j_{0}v)\wt{\dth}+\wt{j_{0}v}) \\
&=-\alpha_{n-1}(j_{0}v)\dth+j_{0}v
\end{align*}
which clearly extends smoothly over the $x_{n}=y_{n}=0$ locus since there is
no $x_{n}$- or $y_{n}$-dependence.
Meanwhile,
a straightforward computation using
\eqref{e:Phi-definition} and
\eqref{e:Phi-pushforward-2}
shows that
\[
\partial_{x_{n}}=\Phi_{*}\bp{e^{-\rho}\bp{\cos\phi\,\drh-\sin\phi\,\dph}} \quad\text{ and }\qquad \partial_{y_{n}}=\Phi_{*}\bp{e^{-\rho}\bp{\sin\phi\,\drh+\cos\phi\,\dph}}
\]
and using this with \eqref{e:Phi-definition}, \eqref{e:Phi-pushforward-1},
and \eqref{e:lift-planes-proof} shows that
\[
-\alpha_{n}(\partial_{x_{n}})\dth+\partial_{x_{n}}
=y_{n}\,\dth+\partial_{x_{n}}
=\Phi_{*}(e^{\rho}\sin\phi\,\wt\dth+e^{-\rho}\cos\phi\,\wt\partial_{\rho})
\]
and
\[
-\alpha_{n}(\partial_{y_{n}})\dth+\partial_{y_{n}}
=-x_{n}\,\dth+\partial_{y_{n}}
=\Phi_{*}(-e^{\rho}\cos\phi\,\wt\dth+e^{-\rho}\sin\phi\,\wt\partial_{\rho}).
\]
Computing further using the definition \eqref{e:j1-definition} of $j_{1}$ with $\tl j_{1}\tl v=\wt{j_{1}v}$
shows that
\[
\Phi_{*}\tl j_{1}(-\alpha_{n}(\partial_{x_{n}})\dth+\partial_{x_{n}})
=-\alpha_{n}(\partial_{y_{n}})\dth+\partial_{y_{n}}
=-x_{n}\,\dth+\partial_{y_{n}}
\]
and
\[
\Phi_{*}\tl j_{1}(-\alpha_{n}(\partial_{y_{n}})\dth+\partial_{y_{n}})
=\alpha_{n}(\partial_{x_{n}})\dth-\partial_{x_{n}}
=-y_{n}\,\dth-\partial_{x_{n}}
\]
which also extend smoothly over $x_{n}=y_{n}=0$.
This completes the proof.
\end{proof}

\bibliographystyle{../hplain5}
\bibliography{nonunique-bib}

\end{document}